\documentclass[11pt,reqno]{article}

\usepackage[utf8]{inputenc}
\usepackage{authblk}
\usepackage{amsmath}
\usepackage{amsthm}
\usepackage{amssymb}
\usepackage{latexsym}
\usepackage{amstext}
\usepackage{tikz}
\usepackage{multicol}

\usepackage{setspace}
\usepackage{parskip}

\usepackage{graphicx}

\usepackage{hyperref}

\theoremstyle{plain}
\newtheorem{theorem}{Theorem}[section]
\newtheorem{corollary}[theorem]{Corollary}

\newtheorem{proposition}[theorem]{Proposition}

\theoremstyle{definition}

\theoremstyle{remark}
\newtheorem*{remark}{Remark}

\title{Elementary numerical methods for double integrals}

\author[1]{Cameron Grant}
\author[2]{Erik Talvila}

\affil[1]{Simon Fraser University\\
	Burnaby, BC Canada V5A 1S6}
\affil[2]{University of the Fraser Valley\\
	Abbotsford, BC Canada V2S 7M8}

\newcommand{\N}{{\mathbb N}}
\newcommand{\R}{{\mathbb R}}

\newcommand{\fn}{\!:\!}

\providecommand{\abs}[1]{\lvert#1\rvert}
\providecommand{\norm}[1]{\lVert#1\rVert}

\newcommand{\intab}{\int_a^b}

\newcommand{\intabcd}{\int_a^b\int_c^d}
\newcommand{\ft}{{\tilde f}}
\newcommand{\alphabar}{{\overline\alpha}}
\newcommand{\betabar}{{\overline\beta}}

\begin{document}

% Keep the maketitle command here
\maketitle

%%%% The abstract must go within the body of the document.
\begin{abstract}
Approximations to the integral $\int_a^b\int_c^d f(x,y)\,dy\,dx$ are obtained under the
assumption that the partial derivatives of the integrand are in an $L^p$ 
space, for some $1\leq p\leq\infty$.  
We assume ${\lVert f_{xy}\rVert}_p$ is bounded (integration over $[a,b]\times[c,d]$),
assume ${\lVert f_x(\cdot,c)\rVert}_p$ and ${\lVert f_x(\cdot,d)\rVert}_p$ are bounded (integration over $[a,b]$),
and assume ${\lVert f_y(a,\cdot)\rVert}_p$ and ${\lVert f_y(b,\cdot)\rVert}_p$ are bounded (integration over $[c,d]$).
The methods are elementary, using only integration by parts and H\"older's inequality.
Versions of the trapezoidal rule, composite trapezoidal rule, midpoint rule and composite
midpoint rule are given, with error estimates in terms of the above norms.
\end{abstract}

\section{Introduction}\label{sectionintroduction}
In this paper, we derive versions of the trapezoidal rule and
midpoint rule for double integrals over finite rectangles.  In order to generate an
error estimate for a quadrature rule, it is necessary to assume
something about the integrand other than mere integrability.
If $f$ is a real-valued function on the rectangle $\Omega=[a,b]\times[c,d]$,
then we give numerical integration formulas for $\intabcd f(x,y)\,dy\,dx$
under the assumption that the mixed partial derivative $f_{xy}$ is
in one of the Lebesgue spaces $L^p(\Omega)$ for some 
$1\leq p\leq\infty$.  (When $p=\infty$, this includes the case
of continuously differentiable $f$.)  We also assume the first order partial
derivatives $f_x$ and $f_y$ are in an $L^p$ space when integrated 
over just $x$ or $y$, respectively.  The methods being presented are elementary, depending only
on H\"older's inequality and integration by parts.

Our results are stated for Lebesgue integrals.  A suitable reference is
\cite{benedetto}.  By considering $f$ to have
continuous second partial derivatives the reader can easily transfer results to
the Riemann integral.

The basis of our method is to take $\phi$ to be a function smooth enough
so that we can carry out integration by parts on $\intabcd f_{xy}(x,y)\phi(x,y)\,dy\,dx$.
If $\phi$ is chosen so that $\phi_{xy}\!=\!1$, then this leads to a formula
relating $\intabcd f(x,y)dydx$ to integrals of derivatives of $f$ multiplied
by $\phi$ or its derivatives  (Proposition~\ref{propparts}).  H\"older's inequality
then gives estimates of the error in terms of $L^p$ norms of $f_x$,
$f_y$, and $f_{xy}$.  Various choices for $\phi$ lead to a double
integral version of the trapezoidal rule, composite trapezoidal rule (Section~\ref{sectiontrapezoidal}),
midpoint rule, and composite midpoint rule (Section~\ref{sectionmidpoint}).  In
Section~\ref{sectionminimize}, we show that when $1<p<\infty$ the unique choice of
$\phi$ that minimizes the error coefficient of $\norm{f_{xy}}_p$ is the same as the choice
that gives the trapezoidal rule.

The literature on one-variable numerical integration is vast; however,
the literature on several-variable numerical integration is sparse.\!
General overviews to the problems of numerical approximation of
multiple integrals are contained in \cite{haber} and \cite{ueberhuber}.
Three sources that use the integration by parts method are Mikeladze \cite{mikeladze}, Sard \cite{sard}, and
Stroud \cite{stroud}.   We extend the results in these papers by considering $f_{xy}\in L^p(\Omega)$
for all $1\leq p\leq \infty$, by computing error estimates, and by establishing conditions under
which the error is minimized.

% Use a blank line to start a new paragraph.
% Avoid manual spacing commands (such as \\ or \vspace). These will be changed in the editing process.

\section{Background}\label{sectionparts}
First we present the basic integration by parts formula that will be used throughout the paper.  Then we
look at minimal conditions under which it holds.  
\begin{proposition}[Integration by Parts]\label{propparts}
	Suppose $f$ and $\phi$ are $C^2$ functions on $[a,b]\times[c,d]$,  then
	\begin{align}
	&\intabcd f(x,y)\phi_{xy}(x,y)\,dy\,dx\label{integral}\\
	&\qquad=f(a,c)\phi(a,c)+f(b,d)\phi(b,d)-f(a,d)\phi(a,d)-f(b,c)\phi(b,c)\label{quad}\\
	&\qquad+\int_a^b\left[f_x(x,c)\phi(x,c)-f_x(x,d)\phi(x,d)\right]dx\label{error1}\\
	&\qquad+\int_c^d\left[f_y(a,y)\phi(a,y)-f_y(b,y)\phi(b,y)\right]dy\label{error2}\\
	&\qquad+\intabcd f_{xy}(x,y)\phi(x,y)\,dy\,dx.\label{error3}
	\end{align}
\end{proposition}
The proposition is proved using integration by parts and the Fubini--Tonelli theorem.   
See Proposition~\ref{propconditions} below for weaker conditions under which it holds.

If we now choose $\phi$ such that $\phi_{xy}=1$, then \eqref{integral}
and \eqref{quad} give a quadrature formula for $\intabcd f(x,y)\,dy\,dx$
with error in 
\eqref{error1}--\eqref{error3}.  To estimate the integrals in the 
error we assume $f_x$, $f_y$ and $f_{xy}$ are in $L^p$ spaces.

What are the solutions of the partial differential equation
$\phi_{xy}=1$?  They are $\phi(x,y)=xy+\alpha(x)+\beta(y)$ where
$\alpha$ and $\beta$ are differentiable functions of one variable.
We will make different choices for $\alpha$ and $\beta$ to derive
trapezoidal and midpoint rules and also to minimize the resulting
error terms.  

Error estimates arise from H\"older's inequality.
We use the $p$-norms
$$
\norm{f}_p=\left(\intabcd\abs{f(x,y)}^p\,dy\,dx\right)^{1/p}
$$ 
for $1\leq p<\infty$ and
$\norm{f}_\infty={\rm ess\,sup}_{(x,y)\in[a,b]\times[c,d]}\abs{f(x,y)}$
in the case $p=\infty$.  This reduces to the maximum of $\abs{f(x,y)}$
when $f$ is continuous.  Also, the one-variable norms for $1\leq p<\infty$ are
$$
\norm{f(\cdot,e_2)}_p=\left(\int_a^b\abs{f(x,e_2)}^p\,dx\right)^{1/p}\!\!
\text{ and }
\norm{f(e_1,\cdot)}_p=\left(\int_c^d\abs{f(e_1,y)}^p\,dy\right)^{1/p}
$$
where $e_2\in[c,d]$ and $e_1\in[a,b]$, with similar definitions when
$p=\infty$.

Denote the absolutely continuous functions on $[a,b]$ by $AC[a,b]$ and
the absolutely continuous functions on $[c,d]$ by $AC[c,d]$.

If $1<p<\infty$, then $p$ and $q$ are conjugate 
exponents if $1/p+1/q=1$.  The pairs $(p,q)=(1,\infty)$ and
$(\infty,1)$ are also conjugate.

\begin{proposition}\label{propholder}
	Suppose $f$ and $\phi$ satisfy the conditions of Proposition~\ref{propconditions} and
	for some $1\leq p\leq\infty$ the
	following norms exist: $\norm{f_{xy}}_p$, $\norm{f_x(\cdot,c)}_p$,
	$\norm{f_x(\cdot,d)}_p$, $\norm{f_y(a,\cdot)}_p$, and
	$\norm{f_y(b,\cdot)}_p$.  Suppose $\phi(x,y)=xy+\alpha(x)+\beta(y)$
	for $\alpha\in AC[a,b]$ and $\beta\in AC[c,d]$.  Then
	\begin{align*}
	&\intabcd f(x,y)\,dy\,dx\\
	&\quad=f(a,c)\phi(a,c)+f(b,d)\phi(b,d)-f(a,d)\phi(a,d)-f(b,c)\phi(b,c)
	+E(f,\phi)
	\end{align*}
	where 
	\begin{align*}
	&E(f,\phi)=\int_a^b\left[f_x(x,c)\phi(x,c)-f_x(x,d)\phi(x,d)\right]dx\\
	&\quad+\int_c^d\left[f_y(a,y)\phi(a,y)-f_y(b,y)\phi(b,y)\right]dy
	+\intabcd f_{xy}(x,y)\phi(x,y)\,dy\,dx
	\end{align*}
	and
	\begin{eqnarray*}
		\abs{E(f,\phi)} & \leq & \norm{f_x(\cdot,c)}_p\norm{\phi(\cdot,c)}_q
		+\norm{f_x(\cdot,d)}_p\norm{\phi(\cdot,d)}_q
		+\norm{f_y(a,\cdot)}_p\norm{\phi(a,\cdot)}_q\\
		& & \quad+\norm{f_y(b,\cdot)}_p\norm{\phi(b,\cdot)}_q
		+\norm{f_{xy}}_p\norm{\phi}_q.
	\end{eqnarray*}
	Here, $p$ and $q$ are conjugate exponents.
\end{proposition}
\begin{proof}
	This follows from Proposition~\ref{propparts}, Proposition~\ref{propconditions}, and H\"older's
	inequality.
\end{proof}

Now we consider weaker conditions under which the formula in Proposition~\ref{propparts}
holds.  Note that the integration by parts formula
$\int_a^bf'(x)\phi(x)\,dx=f(b)\phi(b)-f(a)\phi(a)-\int_a^bf(x)\phi'(x)\,dx$ holds for Lebesgue
integrals when $f$ and $\phi$ are in $AC[a,b]$. See \cite[Theorem~4.6.3]{benedetto}.

If $f_{xy}\in L^1([a,b]\times[c,d])$ and if
$\phi\in L^\infty([a,b]\times[c,d])$, then,
by the Fubini--Tonelli Theorem, the two iterated integrals equal the double integral:
\begin{eqnarray*}
	\int_a^b\left(\int_c^df_{xy}(x,y)\phi(x,y)\,dy\right)dx & = & 
	\int_c^d\left(\int_a^bf_{xy}(x,y)\phi(x,y)\,dx\right)dy\\
	& = &
	\intabcd f_{xy}(x,y)\phi(x,y)\,dy\,dx.
\end{eqnarray*}
From now on we can omit the parentheses in iterated integrals.
We also assume $f\in L^1(\Omega)$.

A sufficient condition for equality $f_{xy}=f_{yx}$ almost everywhere on $\Omega$ is that
$f_x$ and $f_y$ exist on $\Omega$ and $f_{xx}$, $f_{xy}$, $f_{yx}$ and $f_{yy}$ exist
almost everywhere.  This condition is due to Currier \cite{currier}.  Continuity of the 
mixed partial derivatives also ensures their equality everywhere.

For fixed $x\in[a,b]$, we can integrate by parts to get
\begin{align}
&\int_c^df_{xy}(x,y)\phi(x,y)\,dy\label{yfubinicond1}\\ 
&\qquad=f_x(x,d)\phi(x,d)-f_x(x,c)\phi(x,c)
-\int_c^d f_x(x,y)\phi_y(x,y)\,dy.\label{yfubinicond2}
\end{align}
By the Fundamental Theorem of Calculus for Lebesgue integrals,
this holds if
\begin{equation}
f_x(x,\cdot), \phi(x,\cdot)\in AC[c,d]\quad\text{ for almost all } x\in(a,b).\label{AC1}
\end{equation}
We would now like to integrate \eqref{yfubinicond1} and \eqref{yfubinicond2} over
$x\in[a,b]$.  Since $f_{xy}\in L^1(\Omega)$ and $\phi\in L^\infty(\Omega)$, we know
we can do this in \eqref{yfubinicond1}.  Hence, we can also do this in \eqref{yfubinicond2}.
To integrate each term in \eqref{yfubinicond2} separately, we also assume 
$f_{x}\in L^1(\Omega)$ and $\phi_y\in L^\infty(\Omega)$.  We then get
\begin{eqnarray}
\quad\intabcd f_{xy}(x,y)\phi(x,y)\,dy\,dx & = &
\int_a^b\left[f_x(x,d)\phi(x,d)-f_x(x,c)\phi(x,c)\right]dx\notag \\
& & \quad-\intabcd f_x(x,y)\phi_y(x,y)\,dy\,dx.\label{fubinicond2}
\end{eqnarray}
Since $f_{x}\in L^1(\Omega)$ and $\phi_y\in L^\infty(\Omega)$, the Fubini--Tonelli
Theorem allows us to reverse the integration order in \eqref{fubinicond2}.  If
$f(\cdot,y),\phi(\cdot,y)\in AC[a,b]$ for almost all $y\in[c,d]$, then we can
integrate by parts:
\begin{align}
&\intabcd f_x(x,y)\phi_y(x,y)\,dy\,dx
=\int_c^d\int_a^b f_x(x,y)\phi_y(x,y)\,dx\,dy\label{xfubinicond1}\\
&\qquad=\int_c^d[f(b,y)\phi_y(b,y)-f(a,y)\phi_y(a,y)]\,dy \notag\\
&\qquad\qquad-\int_c^d\int_a^b f(x,y)\phi_{xy}(x,y)\,dx\,dy.\label{xfubinicond2}
\end{align}
We have also integrated \eqref{xfubinicond1} and \eqref{xfubinicond2} over $x\in[a,b]$.
This is valid under the assumptions $f\in L^1(\Omega)$ and $\phi_{xy}\in L^\infty(\Omega)$.

These conditions are collected in the following proposition, noting that we could have
performed the initial integration by parts over $x$ instead of over $y$.
\begin{proposition}\label{propconditions}
	Consider the following properties:
	\begin{enumerate}
		\item[(i)]
		$g_x\in L^1(\Omega)$; $g_x(x,\cdot)\in AC[c,d]$ for almost all $x\in[a,b]$;
		$g(\cdot,y)\in AC[a,b]$ for almost all $y\in[c,d]$,
		\item[(ii)]
		$g_y\in L^1(\Omega)$; $g_y(\cdot,y)\in AC[a,b]$ for almost all $y\in[c,d]$;
		$g(x,\cdot)\in AC[c,d]$ for almost all $x\in[a,b]$.
	\end{enumerate}
	Assume $f_x$ and $f_y$ exist on $\Omega$ such that $f_{xx}$, $f_{xy}$, $f_{yx}$, and $f_{yy}$ exist
	almost everywhere.  Then $f_{xy}=f_{yx}$ almost everywhere.  Assume also that $f, f_{xy}\in L^1(\Omega)$ and $\phi, \phi_{xy}\in L^\infty(\Omega)$.  Now suppose if $f$ satisfies (i), then $\phi$ satisfies (ii), and
	if $f$ satisfies (ii), then $\phi$ satisfies (i).
	Then the formula in Proposition~\ref{propparts}
	holds.
\end{proposition}

Note that since our rectangle is finite, we have $L^s\subset L^r$ when $s>r$.  When we write
$\phi(x,y)=xy +\alpha(x)+\beta(y)$, all of the
conditions on $\phi$ are satisfied when $\alpha\in AC[a,b]$ and $\beta\in AC[c,d]$.

If we are willing to use a Riemann--Stieltjes integral, then an integration by
parts formula is 
$\int_a^bf'(x)\phi(x)\,dx=f(b)\phi(b)-f(a)\phi(a)-\int_a^bf(x)\,d\phi(x)$, provided
$f$ is continuous and 
$\phi$ is of bounded variation.  There is a related formula when $f$ is merely regulated, i.e.
it has left and right limits at each point.  See \cite{mcleod}.  With this formulation,
the conditions on $f$ in Proposition~\ref{propconditions} can be weakened as long as
the conditions on $\phi$ are suitably strengthened.

\section{Trapezoidal Rule}\label{sectiontrapezoidal}
For a function of one variable, a trapezoidal rule is
$\int_a^b g(x)\,dx =[g(a)+g(b)](b-a)/2 +E(g)$, where
$E(g)=-\int_a^b g'(x)(x-c)\,dx$, and $c$ is the midpoint of $[a,b]$.  
This follows from
integration by parts. See \cite[Theorem~1.8]{cruzuribeJIPAM}.
H\"older's inequality, then, gives the estimate
$$
\abs{E(g)}  \leq  \left\{\begin{array}{cl}
\frac{1}{2}\norm{g'}_1(b-a), & p=1\\
\frac{1}{2}\left(q+1\right)^{-1/q}\norm{g'}_p
(b-a)^{1+1/q}, & 
1<p<\infty\\
\frac{1}{4}\norm{g'}_\infty(b-a)^2, & p=\infty,
\end{array}
\right.
$$
where, again, $p$ and $q$ are conjugate exponents.
The estimate is sharp in the sense that the coefficients of the
norms cannot be reduced.
The paper \cite{talvilawiersmaaejm} shows an integration by parts
method that can be used to derive the usual trapezoidal rule when it
is assumed $g''$ is bounded.

For a function of two variables, we choose $\phi$ so that
$f$ is evaluated at the four corners of the rectangle $[a,b]\times [c,d]$.  %Remove "equally weighted nodes" because of (11). Doesn't seem necessary to mention it.
For this we let $m_1$ be the midpoint of
$[a,b]$, let $m_2$ be the midpoint of $[c,d]$, and take $\phi(x,y)=
(x-m_1)(y-m_2)=xy-m_2x-m_1y+m_1m_2$ so that 
$\alpha(x)=-m_2x +m_1m_2$ and $\beta(y)=-m_1y$.

\begin{theorem}[Trapezoidal Rule]\label{theoremtrapezoidal}
	Suppose $f$ satisfies the conditions of Proposition~\ref{propconditions}, and
	for some $1\leq p\leq\infty$ the
	following norms exist:
	$\norm{f_{xy}}_p$, $\norm{f_x(\cdot,c)}_p$,
	$\norm{f_x(\cdot,d)}_p$, $\norm{f_y(a,\cdot)}_p$, and
	$\norm{f_y(b,\cdot)}_p$.
	Then we have that
	$$
	{\int_a^b\!\int_c^d} \!\!f(x,y)\,dy\,dx\!=\!\left[f(a,c)\!+\!f(b,d)\!+\!f(a,d)\!+\!f(b,c)\right]
	\frac{(b-a)(d-c)}{4}+E(f).
	$$
	If $p=1$, then
	\begin{eqnarray*}
		\abs{E(f)} & \leq &
		\left(\norm{f_x(\cdot,c)}_1
		+\norm{f_x(\cdot,d)}_1\right)\frac{
			(b-a)(d-c)}{4}\\
		& & \quad+\left(\norm{f_y(a,\cdot)}_1
		+\norm{f_y(b,\cdot)}_1\right)\frac{(b-a)(d-c)}{4}\\
		& & \quad +\frac{\norm{f_{xy}}_1(b-a)(d-c)}{4}.
	\end{eqnarray*}
	
	If $1<p<\infty$, then
	\begin{eqnarray*}
		\abs{E(f)} & \leq &
		\left(\norm{f_x(\cdot,c)}_p
		+\norm{f_x(\cdot,d)}_p\right)\frac{(d-c)
			(b-a)^{2-1/p}}{4}\left(\frac{p-1}{2p-1}
		\right)^{1-1/p}\\
		& & \quad+\left(\norm{f_y(a,\cdot)}_p
		+\norm{f_y(b,\cdot)}_p\right)\frac{(b-a)(d-c)^
			{2-1/p}}{4}\left(\frac{p-1}{2p-1}
		\right)^{1-1/p}\\
		& & \quad+\frac{\norm{f_{xy}}_p(b-a)^{2-1/p}(d-c)^{2-1/p}}{4}\left(\frac{p-1}{2p-1}\right)^{2(1-1/p)}.
	\end{eqnarray*}
	
	If $p=\infty$, then
	\begin{eqnarray*}
		\abs{E(f)} & \leq &
		\left(\norm{f_x(\cdot,c)}_\infty
		+\norm{f_x(\cdot,d)}_\infty\right)\frac{
			(b-a)^2(d-c)}{8}\\
		& & \quad+\left(\norm{f_y(a,\cdot)}_\infty
		+\norm{f_y(b,\cdot)}_\infty\right)\frac{(b-a)(d-c)^2}{8}\\
		& & \quad +\frac{\norm{f_{xy}}_\infty(b-a)^2(d-c)^2}{16}.
	\end{eqnarray*}
\end{theorem}
\begin{proof}
	Putting $\phi(x,y)=(x-m_1)(y-m_2)$ into
	Proposition~\ref{propholder} yields the quadrature formula.
	
	Let $\psi(t)=t$.  Compute the norms of $\psi$ over $[-1,1]$.
	If $1\leq q<\infty$, then
	$$
	\norm{\psi}_q  =  \left(\int_{-1}^1\abs{t}^q\,dt\right)^{1/q}
	=  \left(2\int_{0}^1t^q\,dt\right)^{1/q} =\left(\frac{2}{
		q+1}\right)^{1/q}.
	$$
	If $q=\infty$, we have
	$$
	\norm{\psi}_\infty=\max_{\abs{t}\leq 1}\abs{t}=1.
	$$
	H\"older's inequality and a linear change of variables give
	$$
	\left|\int_a^bf_x(x,c)\phi(x,c)dx\right|
	\leq \norm{f_x(\cdot,c)}_p
	\left(\int_a^b\abs{x-m_1}^q\,dx\right)^{1/q}\frac{(d-c)}{2}.
	$$
	Note that
	$$
	\left(\int_a^b\abs{x-m_1}^q\,dx\right)^{1/q}\!\!\! = 
	\left(\int_{a-m_1}^{b-m_1}\!\!\!\!\!\!\!\!\!\!\abs{x}^q\,dx\right)^{1/q}
	\!\!\!=  \norm{\psi}_q\left(\frac
	{b-a}{2}\right)^{1+1/q}
	\!\!\!=  \frac
	{(b-a)^{1+1/q}}{2(q+1)^{1/q}}.
	$$
	If $p=1$ we have
	$$
	\max_{a\leq x\leq b}\abs{x-m_1}=
	\max_{a-m_1\leq x\leq b-m_1}\abs{x}=\max_{\abs{t}\leq(b-a)/2}
	\abs{\psi(t)}=\norm{\psi}_\infty\frac{b-a}{2}=\frac{b-a}{2}.
	$$
	If we observe that, for $1<p\leq\infty$,
	\begin{eqnarray*}
		\abs{E(f)} & \leq &
		\left(\norm{f_x(\cdot,c)}_p
		+\norm{f_x(\cdot,d)}_p\right)
		\norm{\psi}_q
		\left(\frac{b-a}{2}\right)^{1+1/q}
		\left(\frac{d-c}{2}\right)\\
		& & \quad+\left(\norm{f_y(a,\cdot)}_p
		+\norm{f_y(b,\cdot)}_p\right)\norm{\psi}_q\left(\frac{b-a}{2}\right)
		\left(\frac{d-c}{2}\right)^{1+1/q}\\
		& & \quad
		+\norm{f_{xy}}_p\norm{\psi}_q\left(\frac{b-a}{2}\right)^{1+1/q}\left(\frac{d-c}{2}\right)
		^{1+1/q},
	\end{eqnarray*}
	
	then the result follows upon writing $q$ in terms of
	$p$. For $p=1$, take the limit of the above expression
	as $q\to\infty$.
\end{proof}
\begin{corollary}
	If $\abs{\nabla f}\leq M$ and $\abs{f_{xy}}\leq N$ for
	some $M, N\in \R$, then
	$$
	\abs{E(f)}\leq \frac{M(b-a)^2(d-c)}{4}
	+\frac{M(b-a)(d-c)^2}{4}
	+\frac{N(b-a)^2(d-c)^2}{16}.
	$$
\end{corollary}

\begin{corollary}[Trapezoidal Composite Rule]\label{trapcomprule}
	Define a uniform partition of $[a,b]$ by $x_i=a+i\Delta x$ where
	$\Delta x=(b-a)/m$ for some $m\in\N$.  Then, for $0\leq i\leq m$,
	we have $a=x_0<x_1<\ldots<x_m=b$.
	Define a uniform partition of $[c,d]$ by $y_j=c+j\Delta y$ where
	$\Delta y=(d-c)/n$ for some $n\in\N$.  Then, for $0\leq j\leq n$,
	we have $c=y_0<y_1<\ldots<y_n=d$.
	Then 
	\begin{align*}
	&\int_a^b\int_c^df(x,y)\,dy\,dx=
	\Bigg[f(a,c)+f(b,d)+f(a,d)+f(b,c)\\
	&\quad+\!2\sum_{j=1}^{n-1}\!f(a,y_j)\!+\!2\sum_{j=1}^{n-1}\!f(b,y_j)
	\!+\!2\sum_{i=1}^{m-1}\!f(x_i,c)\!+\!2\sum_{i=1}^{m-1}\!f(x_i,d)\Bigg]\!\frac{(b\!-\!a)(d\!-\!c)}{4mn}\\
	&\quad+E(f).
	\end{align*}
	If $p=1$, then
	\begin{align*}
	&\abs{E(f)}\leq\\
	&\left(\norm{f_x(\cdot,c)}_1+2\sum_{j=1}^n\norm{f_x(\cdot,y_j)}_1
	+\norm{f_x(\cdot,d)}_1\right)\frac{(b-a)(d-c)}{4mn}\\
	&+\left(\norm{f_y(a,\cdot)}_1+2\sum_{i=1}^m\norm{f_y(x_i,\cdot)}_1
	+\norm{f_y(b,\cdot)}_1\right)\frac{(b-a)(d-c)}
	{4mn}\\
	&+\frac{\norm{f_{xy}}_1(b-a)(d-c)}{4mn}.
	\end{align*}
	If $1<p<\infty$, then
	\begin{align*}
	&\abs{E(f)}\leq\\
	&\left(\norm{f_x(\cdot,c)}_p\!+\!2\sum_{j=1}^n\norm{f_x(\cdot,y_j)}_p
	\!+\!\norm{f_x(\cdot,d)}_p\right)\!\frac{(d\!-\!c)
		(b\!-\!a)^{2\!-\!1/p}}{4mn}\!\left(\frac{p\!-\!1}{2p\!-\!1}
	\right)^{1\!-\!1/p}\\
	&\!+\!\left(\!\!\norm{f_y(a,\!\cdot)}_p\!\!+\!\!2\sum_{i=1}^m\norm{f_y(x_i,\!\cdot)}_p
	+\norm{f_y(b,\!\cdot)}_p\!\right)\!\!\frac{(b-a)(d-c)
		^{2-1/p}}{4mn}\!\!\left(\frac{p\!-\!1}{2p\!-\!1}
	\right)\!\!^{1\!-\!1/p}\\
	&+\frac{\norm{f_{xy}}_p(b-a)^{2-1/p}(d-c)^{2-1/p}}{4mn}\left(\frac{p-1}{2p-1}\right)^{2(1-1/p)}.
	\end{align*}
	If $p=\infty$, then
	\begin{align*}
	&\abs{E(f)}\leq\\
	&\left(\norm{f_x(\cdot,c)}_\infty+2\sum_{j=1}^n\norm{f_x(\cdot,y_j)}_\infty
	+\norm{f_x(\cdot,d)}_\infty\right)\frac{(d-c)
		(b-a)^{2}}{8mn}\\
	&+\left(\norm{f_y(a,\cdot)}_\infty+2\sum_{i=1}^m\norm{f_y(x_i,\cdot)}_\infty
	+\norm{f_y(b,\cdot)}_\infty\right)\frac{(b-a)(d-c)
		^{2}}{8mn}\\
	&+\frac{\norm{f_{xy}}_\infty(b-a)^{2}(d-c)^{2}}{16mn}.
	\end{align*}
	
	If $\abs{\nabla f}\leq M$ and $\abs{f_{xy}}\leq N$ for
	some $M, N\in \R$, then
	\begin{eqnarray*}
		\abs{E(f)} & \leq
		&
		\frac{M(2n+1)(d-c)
			(b-a)^{2}}{8mn}\\
		& &\quad +
		\frac{M(2m+1)(b-a)(d-c)
			^{2}}{8mn}\\
		& &\quad + \frac{N(b-a)^{2}(d-c)^{2}}{16mn}.
	\end{eqnarray*}
	%Do we need more in this sentence?
	%See revision number 16.
	Note that $(2n+1)/n\leq 3$ and $(2n+1)/n\sim 2$ as $n\to\infty$.
\end{corollary}

\begin{proof}
	To obtain the integral approximation, define $\phi(x,y)=U_i(x)V_j(y)$ where $U_i(x)=(x-u_i)$ when $x\in(x_{i-1},x_{i})$
	for some $1\leq i\leq m$ and $U_i=0$ otherwise and $V_j(y)=(y-v_j)$ when $y\in(y_{j-1},y_j)$ for some $1\leq j\leq n$ and
	$V_j=0$ otherwise. Here, $u_i=(x_{i-1}+x_{i})/2$
	and $v_j=(y_{j-1}+y_j)/2$.  Now write
	$$
	\intabcd f(x,y)\,dy\,dx=\sum_{i=1}^m\sum_{j=1}^n\int_{x_{i-1}}^{x_i}\int_{y_{j-1}}^{y_j}f(x,y)\,dy\,dx,
	$$
	and apply Proposition~\ref{propholder} to each term in the sum.
	
	The error becomes
	\begin{eqnarray*}
		E(f) & = &
		-\sum_{i=1}^m\sum_{j=1}^n\left\{\int_{x_{i-1}}^{x_i}\left[f_x(x,y_{j-1})+f_x(x,y_j)\right]U_i(x)\,dx\,\frac{\Delta y}{2}
		\right.\\
		& &\quad -\int_{y_{j-1}}^{y_j}\left[f_y(x_{i-1},y)+f_y(x_i,y)\right]V_j(y)\,dy\,\frac{\Delta x}{2}\\
		& &\quad +\left.\int_{x_{i-1}}^{x_i}\int_{y_{j-1}}^{y_j}f_{xy}(x,y)U_i(x)V_j(y)\,dy\,dx\right\}\\
		& = & -\int_a^b\left[f_x(x,c)+2\sum_{j=1}^nf_x(x,y_j)+f_x(x,d)\right]U_i(x)\,dx\,\frac{\Delta y}{2}\\
		&  &\quad -\int_c^d\left[f_y(a,y)+2\sum_{i=1}^mf_y(x_i,y)+f_y(b,y)\right]V_j(y)\,dy\,
		\frac{\Delta x}{2}\\
		&  & +\quad \int_a^b\int_c^df_{xy}(x,y)U_i(x)V_j(y)\,dy\,dx.
	\end{eqnarray*}
	
	The error estimate, then, follows as in the theorem.  We can take limits as $p\to 1$ or $p\to\infty$ as in the theorem.
\end{proof}

\section{Midpoint Rule}\label{sectionmidpoint}
The midpoint rule for a function of one variable is
$\int_a^b g(x)\,dx=g(m)(b-a) +E(g)$, where $m$ is the
midpoint of interval $[a,b]$,
$E(g)=-\int_a^b g'(x)\omega(x)\,dx$, $\omega(x)=
x-a$ for $a\leq x<m$, and $\omega(x)=x-b$ for $m<x\leq b$.
This follows upon integration by parts.  Since it is used
only for integration, the value of $\omega$ at $m$ is
irrelevant.  Notice that $\omega(a)=\omega(b)=0$, $\omega$ has a jump discontinuity
at $m$, and $\omega'(x)=1$ for all $x\not=m$.\\

To construct a midpoint rule when integrating over
$[a,b]\times[c,d]$, look at the formulas in Proposition~\ref{propholder}.  We would like to choose $\phi$ to vanish on the
boundary of the rectangle.  As in the one-variable problem,
this can be done with a piecewise definition.
\begin{theorem}[Midpoint Rule]\label{theoremmidpoint}
	Suppose $f$ satisfies the conditions of Proposition~\ref{propconditions} and
	for some $1\leq p\leq\infty$ the
	following norms exist:
	$\norm{f_{xy}}_p$, $\norm{f_x(\cdot,c)}_p$,
	$\norm{f_x(\cdot,d)}_p$, $\norm{f_y(a,\cdot)}_p$, and
	$\norm{f_y(b,\cdot)}_p$.
	Let $m_1$ be the midpoint of $[a,b]$ and $m_2$
	be the midpoint of $[c,d]$.
	Then  
	$$
	\intabcd f(x,y)\,dy\,dx=f(m_1,m_2)
	(b-a)(d-c) +E(f).
	$$
	If $p=1$, then
	\begin{eqnarray*}
		\abs{E(f)} & \leq &
		\norm{f_x(\cdot,m_2)}_1\frac{
			(b-a)(d-c)}{2}\\
		& & \quad+\norm{f_y(a,\cdot)}_1
		\frac{(b-a)(d-c)}{2}\\
		& & \quad +\frac{\norm{f_{xy}}_1(b-a)(d-c)}{4}.
	\end{eqnarray*}
	
	If $1<p<\infty$, then
	\begin{eqnarray*}
		\abs{E(f)} & \leq &
		\norm{f_x(\cdot,m_2)}_p
		\frac{(d-c)
			(b-a)^{2-1/p}}{2}\left(\frac{p-1}{2p-1}
		\right)^{1-1/p}\\
		& & \quad+\norm{f_y(m_1,\cdot)}_p
		\frac{(b-a)(d-c)^
			{2-1/p}}{2}\left(\frac{p-1}{2p-1}
		\right)^{1-1/p}\\
		& & \quad+\frac{\norm{f_{xy}}_p(b-a)^{2-1/p}(d-c)^{2-1/p}}{4}\left(\frac{p-1}{2p-1}\right)^{2(1-1/p)}.
	\end{eqnarray*}
	
	If $p=\infty$, then
	\begin{eqnarray*}
		\abs{E(f)} & \leq &
		\norm{f_x(\cdot,m_2)}_\infty\frac{
			(b-a)^2(d-c)}{4}\\
		& & \quad+\norm{f_y(m_1,\cdot)}_\infty
		\frac{(b-a)(d-c)^2}{4}\\
		& & \quad +\frac{\norm{f_{xy}}_\infty(b-a)^2(d-c)^2}{16}.
	\end{eqnarray*}
\end{theorem}
\begin{figure}
{\includegraphics
%[scale=.68]
[width=4.988in]
{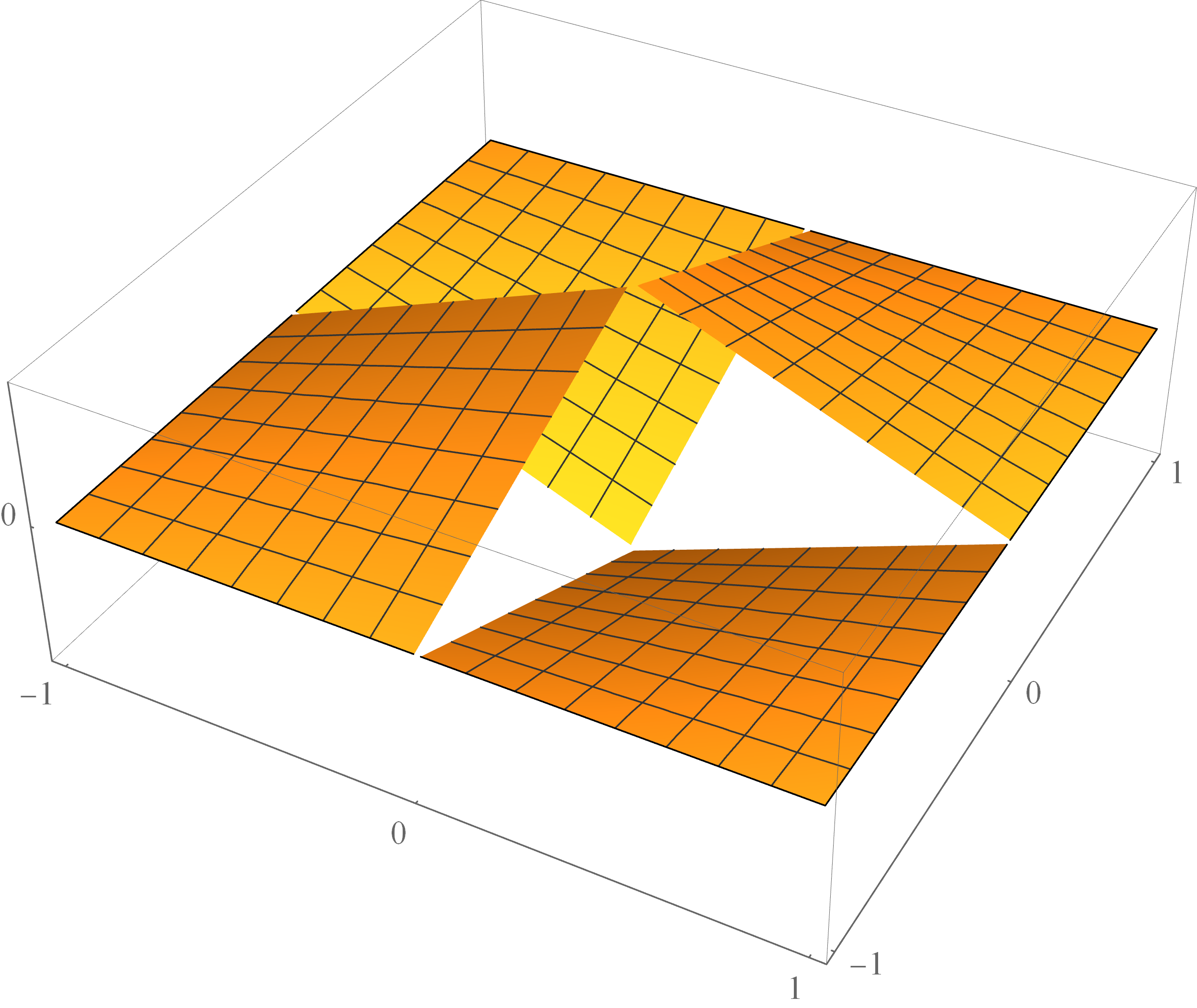}}
\caption{Midpoint rule $\phi$.\label{midpointphi}}

\end{figure}

\begin{proof}
	It is simplest to first solve the normalized problem
	when $[a,b]\times[c,d]=[-1,1]\times[-1,1]$ and the
	function to be integrated is $\ft$.
	Define 
	$$
	\phi(s,t)=\left\{\begin{array}{lll}
	(s-1)(t-1); & 0< s\leq 1, & 0< t\leq 1\\
	(s+1)(t-1); & -1\leq s<0, & 0< t\leq 1\\
	(s+1)(t+1); & -1\leq s<0, & -1\leq  t<0\\
	(s-1)(t+1); & 0<s\leq 1, & -1\leq  t<0.
	\end{array}
	\right.
	$$

	See Figure~\ref{midpointphi} for a plot of $\phi$.

	Consider integration in the region $[0,1]\times[0,1]$.
	Using Proposition~\ref{propparts},
	\begin{align*}
	&\int_{0}^1\int_{0}^1 \ft(s,t)\,dt\,ds  =
	\ft(0,0)
	-\int_{0}^1\ft_s(s,0)(s-1)\,ds\\
	&\qquad-\int_{0}^1 \ft_t(0,t)(t-1)\,dt
	+\int_{0}^1\int_{0}^1 \ft_{st}(s,t)(s-1)
	(t-1)\,dt\,ds.
	\end{align*}
	There are similar formulas for the other three regions.
	We can then define
	$\gamma(x)=x+1$ for $x<0$ and $\gamma(x)=x-1$ for
	$x>0$.  Next we have
	\begin{align*}
	&\int_{-1}^1\int_{-1}^1 \ft(s,t)\,dt\,ds =4\ft(0,0)
	-2\int_{-1}^1 \ft_s(s,0)\gamma(s)\,ds\\
	&\qquad-2\int_{-1}^1 \ft_t(0,t)\gamma(t)\,dt
	+\int_{-1}^1\int_{-1}^1 \ft_{st}(s,t)\gamma(s)\gamma(t)\,
	dt\,ds.
	\end{align*}
	This gives 
	\begin{equation}
	\int_{-1}^1\int_{-1}^1 \ft(s,t)\,dt\,ds =4\ft(0,0)+E(\ft),
	\label{f(s,t)}
	\end{equation}
	where
	\begin{eqnarray*}
		E(\ft) & = & -2\int_{-1}^1 \ft_s(s,0)\gamma(s)\,ds
		-2\int_{-1}^1 \ft_t(0,t)\gamma(t)\,dt\\
		& & \quad+\int_{-1}^1\int_{-1}^1 \ft_{st}(s,t)\gamma(s)\gamma(t)\,
		dt\,ds.
	\end{eqnarray*}
	H\"older's inequality shows
	\begin{equation}
	\abs{E(\ft)}  \leq  2\norm{\ft_s(\cdot,0)}_p\norm{\gamma}_q+
	2\norm{\ft_t(0,\cdot)}_p\norm{\gamma}_q+
	\norm{\ft_{st}}_p\norm{\gamma}_q^2,\label{holderst}
	\end{equation}
	where the norms are now taken over $[-1,1]$ and 
	$[-1,1]\times[-1,1]$.
	
	Note that if $1\leq q<\infty$, then
	$$
	\norm{\gamma}_q 
	=  \left(\int_{-1}^0(1+s)^qds
	+\! \int_{0}^1(1-s)^qds
	\right)^{1/q}
	\!\!\!=  \left(2\!\int_{0}^1u^qdu\right)^{1/q}
	\!\!\!=  \left(\frac{2}{q+1}\right)^{1/q},
	$$
	and, $\norm{\gamma}_\infty=1$.
	
	The transformation $x=(b-a)s/2+m_1$ and 
	$y=(d-c)t/2+m_2$, maps the unit square onto
	$[a,b]\times[c,d]$.  Let $\ft(s,t)=f(x,y)$.
	In \eqref{f(s,t)},
	\begin{equation}
	\int_{-1}^1\int_{-1}^1 \ft(s,t)\,dt\,ds =\frac{4}{(b-a)(d-c)}
	\intabcd f(x,y)\,dy\,dx.
	\label{f(x,y)}
	\end{equation}
	For $1\leq p<\infty$, we also have
	\begin{eqnarray}
	\norm{\ft_s(\cdot,0)}_p & = & \left(\int_{-1}^1\abs{\ft_s(s,0)}^p
	\,ds\right)^{1/p}
	=  \left(\int_{a}^b\left|\frac{\partial f(x,m_2)}{\partial x}
	\frac{dx}{ds}\right|^p\,\frac{ds}{dx}\,dx\right)^{1/p}\notag\\
	& = & \norm{f_x(\cdot,m_2)}_p\left(\frac{dx}{ds}\right)^{1-1/p}
	=  \norm{f_x(\cdot,m_2)}_p\left(\frac{b-a}{2}\right)^{1-1/p}.
	\label{fsnorm}
	\end{eqnarray}
	And, 
	$$
	\norm{\ft_s(\cdot,0)}_\infty =\max_{\abs{s}\leq 1}
	\left|\ft_s(s,0)\right|=\max_{a\leq x\leq b}\left|f_x(x,m_2)
	\frac{dx}{ds}\right|=\norm{f_x(\cdot,m_2)}_\infty(b-a)/2.
	$$
	The other norms in \eqref{holderst} are handled similarly.
	
	Now putting \eqref{fsnorm} and these other results into
	\eqref{f(x,y)} and \eqref{holderst} gives the formulas
	in the theorem.
\end{proof}
\begin{corollary}[Midpoint Composite Rule]\label{cormidcomposite}
	Define a uniform partition of $[a,b]$ by $x_i=a+i\Delta x$ where
	$\Delta x=(b-a)/m$ for some $m\in\N$.  Then, for $0\leq i\leq m$,
	we have $a=x_0<x_1<\ldots<x_m=b$.
	Define a uniform partition of $[c,d]$ by $y_j=c+j\Delta y$ where
	$\Delta y=(d-c)/n$ for some $n\in\N$.  Then, for $0\leq j\leq n$,
	we have $c=x_0<y_1<\ldots<y_n=d$.  Let $m_i$ be the midpoint of $[x_{i-1},x_i]$,
	and $n_j$ be the midpoint of $[y_{j-1},y_j]$.  We write
	$$
	\intabcd f(x,y)\,dy\,dx=\sum_{i=1}^m\sum_{j=1}^n f(m_i,n_j)\frac{(b-a)(d-c)}{mn}+E(f).
	$$
	If
	$p=1$, then
	\begin{eqnarray*}
		\abs{E(f)} & \leq & \sum_{j=1}^n\norm{f_x(\cdot,n_j)}_1\frac{(b-a)(d-c)}{2n}
		+\sum_{i=1}^m\norm{f_y(m_i,\cdot)}_1\frac{(b-a)(d-c)}{2m}\\
		&   & \quad +\frac{\norm{f_{xy}}_1(b-a)(d-c)}{4}.
	\end{eqnarray*}
	If
	$1<p<\infty$, then
	\begin{eqnarray*}
		\abs{E(f)} & \leq & \sum_{j=1}^n\norm{f_x(\cdot,n_j)}_p\frac{(b-a)^{2-1/p}(d-c)}{2m^{1-1/p}n}
		+\sum_{i=1}^m\norm{f_y(m_i,\cdot)}_p\frac{(b-a)(d-c)^{2-1/p}}{2mn^{1-1/p}}\\
		&  & \quad +\frac{\norm{f_{xy}}_p(b-a)^{2-1/p}(d-c)^{2-1/p}}{4(mn)^{1-1/p}}\left(\frac{p-1}{2p-1}\right)^{2(1-1/p)}.
	\end{eqnarray*}
	If
	$p=\infty$, then
	\begin{eqnarray*}
		\abs{E(f)} & \leq & \sum_{j=1}^n\norm{f_x(\cdot,n_j)}_\infty\frac{(b-a)^2(d-c)}{4mn}
		+\sum_{i=1}^m\norm{f_y(m_i,\cdot)}_\infty\frac{(b-a)(d-c)^2}{4mn}\\
		&   & \quad +\frac{\norm{f_{xy}}_\infty(b-a)^2(d-c)^2}{4mn}.
	\end{eqnarray*}
	If $\abs{\nabla f}\leq M$ and $\abs{f_{xy}}\leq N$ for
	some $M, N\in \R$, then
	\begin{eqnarray*}
		\abs{E(f)} & \leq
		& \frac{M(b-a)^2(d-c)}{4m}+\frac{M(b-a)(d-c)^2}{4n}+\frac{N(b-a)^2(d-c)^2}{4mn}.
	\end{eqnarray*}
	
\end{corollary}

\begin{proof}
	Define
	$$
	\phi(x,y)=\left\{\begin{array}{cl}
	(x-x_i)(y-y_j), & (x,y)\in(m_i,x_i)\times(n_j,y_j)\\
	(x-x_{i-1})(y-y_j), & (x,y)\in(x_{i-1},m_i)\times(n_j,y_j)\\
	(x-x_{i-1})(y-y_{j-1}), & (x,y)\in(x_{i-1},m_i)\times(y_{j-1},n_j)\\
	(x-x_i)(y-y_{j-1}), & (x,y)\in(m_i,x_i)\times(y_{j-1},n_j),
	\end{array}
	\right.
	$$
	where 
	$$
	\gamma_i(x)=\left\{\begin{array}{cl}
	x-x_i, & \text{ if } x\in(m_i,x_i) \text{ for some } 1\leq i\leq m\\
	x-x_{i-1}, & \text{ if } x\in(x_{i-1},m_i)\text{ for some } 1\leq i\leq m\\
	0, & \text{ otherwise,}
	\end{array}
	\right.
	$$
	$$
	\delta_j(y)=\left\{\begin{array}{cl}
	y-y_j, & \text{ if } y\in(n_j,y_y) \text{ for some } 1\leq j\leq n\\
	y-y_{j-1}, & \text{ if } y\in(y_{j-1},n_j)\text{ for some } 1\leq j\leq n\\
	0, & \text{ otherwise.}
	\end{array}
	\right.
	$$
	
	Applying Proposition~\ref{propparts} to each of the four regions gives
	\begin{eqnarray}
	\int_{x_{i-1}}^{x_i}\int_{y_{j-1}}^{y_j}f(x,y)\,dy\,dx\!\!\! & =\!\!\! & \frac{f(x_i,y_j)(b-a)(d-c)}{mn}\label{midcomp1}\\
	&  &\quad- \frac{d-c}{n}\int_{x_{i-1}}^{x_i}f_x(x,n_j)\gamma_i(x)\,dx\label{midcomp2}\\
	&  &\quad- \frac{b-a}{m}\int_{y_{j-1}}^{y_y}f_y(m_i,y)\delta_j(y)\,dy\label{midcomp3}\\
	&  &\quad+\! \int_{x_{i-1}}^{x_i}\int_{y_{j-1}}^{y_j}f_{xy}(x,y)\gamma_i(x)\delta_j(y)\,dy\,dx.\label{midcomp4}
	\end{eqnarray}
	Summing over $i$ and $j$, \eqref{midcomp1} gives the integral approximation.
	
	Let $\chi_I$ be the characteristic function of interval $I$, that is $\chi_I(x)=1$ if $x\in I$ and $0$ otherwise.
	
	From \eqref{midcomp2}, with H\"older's inequality,
	\begin{align*}
	&\frac{d-c}{n}\left|\sum_{i=1}^m\sum_{j=1}^n\int_{x_{i-1}}^{x_i}f_x(x,n_j)\gamma_i(x)\,dx\right|\\
	&\qquad\leq\frac{d-c}{n}\sum_{j=1}^n\int_{a}^{b}\abs{f_x(x,n_j)\gamma_i(x)\chi_{(x_{i-1},x_i)}(x)}\,dx\\
	&\qquad\leq\frac{d-c}{n}\sum_{j=1}^n\norm{f_x(\cdot,n_j)}_p\norm{\gamma_i(x)\chi_{(x_{i-1},x_i)}}_q.
	\end{align*}
	Note that
	\begin{eqnarray*}
		\norm{\gamma_i(x)\chi_{(x_{i-1},x_i)}}_q & = & \left(\sum_{i=1}^m\int_{x_{i-1}}^{m_i}\abs{x-x_{i-1}}^q\,dx
		+\int_{m_{i}}^{x_i}\abs{x-x_{i}}^q\,dx\right)^{1/q}\\
		& = & \left(2\sum_{i=1}^m\int_0^{\Delta x/2}x^q\,dx\right)^{1/q}\\
		& = & \left\{\begin{array}{cl}
			\frac{b-a}{2}, & p=1\\
			\frac{
				(b-a)^{2-1/p}}{2m^{1-1/p}}\left(\frac{p-1}{2p-1}
			\right)^{1-1/p}, & 1<p<\infty\\
			\frac{(b-a)^2}{4m}, & p=\infty.
		\end{array}
		\right.
	\end{eqnarray*}
	
	Equation \eqref{midcomp3} is handled similarly.
	
	With \eqref{midcomp4} we let $\Gamma(x,y)=\gamma_i(x)\delta_j(y)$ if $(x,y)\in(x_{i-1},x_{i})\times
	(y_{j-1},y_j)$ for some $i$ and $j$, and $\Gamma$ is zero otherwise.  Then
	\begin{eqnarray*}
		\left|\sum_{i=1}^m\sum_{j=1}^n\int_{x_{i-1}}^{x_i}\int_{y_{j-1}}^{y_j}f_{xy}(x,y)\gamma_i(x)\delta_j(y)\,dy\,dx\right|
		& \leq & \intab\abs{f_{xy}(x,y)\Gamma(x,y)}\,dy\,dx\\
		& \leq & \norm{f_{xy}}_p\norm{\Gamma}_q,
	\end{eqnarray*}
	and
	\begin{eqnarray*}
		\norm{\Gamma}_q & = & \left(\sum_{i=1}^m\sum_{j=1}^n\int_{x_{i-1}}^{x_i}\abs{\gamma_i(x)}^q\,dx\int_{y_{j-1}}^{y_j}
		\abs{\delta_j(y)}^q\,dy\right)^{1/q}\\
		& = & \left(4\sum_{i=1}^m\sum_{j=1}^n\int_0^{\Delta x/2}x^q\,dx\int_0^{\Delta y/2}y^q\,dy\right)^{1/q}\\
		& = & \left\{\begin{array}{cl}
			\frac{(b-a)(d-c)}{4}, & p=1\\
			\frac{
				(b-a)^{2-1/p}(d-c)^{2-1/p}}{4(mn)^{1-1/p}}\left(\frac{p-1}{2p-1}
			\right)^{2(1-1/p)}, & 1<p<\infty\\
			\frac{(b-a)^2(d-c)^2}{4mn}, & p=\infty.
		\end{array}
		\right.
	\end{eqnarray*}
\end{proof}

\begin{remark}
	At the end of Corollaries~\ref{trapcomprule} and \ref{cormidcomposite}
	we have estimates for the error in the trapezoidal and midpoint composite
	rules  under the assumptions
	$\abs{\nabla f}\leq M$ and $\abs{f_{xy}}\leq N$ for
	some $M, N\in \R$.  If we take partitions with equal number of intervals
	in the $x$ and $y$ direction ($m=n$) then the error estimates for both
	composite rules are $E(f)=O(1/n)$ as $n\to\infty$.
	
	Note that only under the assumptions that $\norm{f_x(\cdot,y)}_p$ is uniformly bounded
	for $c\leq y\leq d$ and $\norm{f_y(x,\cdot)}_p$ is uniformly bounded for $a\leq x\leq b$
	the trapezoidal rule has a better error estimate ($E(f)=O(1/n)$) than the
	midpoint rule ($E(f)=O(1/n^{1-1/p})$).
\end{remark}

\section{Minimizing error estimates}\label{sectionminimize}
The error estimate in Proposition~\ref{propholder} depends on $\norm{\phi}_q$,
where $\phi(x,y)=xy+\alpha(x)+\beta(y)$.  We needed to choose particular
functions $\alpha$ and $\beta$ to generate the
trapezoidal rule (Theorem~\ref{theoremtrapezoidal}) and the
midpoint rule (Theorem~\ref{theoremmidpoint}).  A natural question is:
how can $\alpha$ and $\beta$ be chosen to minimize $\norm{\phi}_q$?
As we see below, if $1<q<\infty$, there is a unique function of this type
that minimizes the norm
of $\phi$ and this is the same $\phi$ as in the trapezoidal rule of 
Theorem~\ref{theoremtrapezoidal}.  If $q=\infty$
the minimizer is not unique but the minimum norm is the same as in the
trapezoidal rule.  For $q=1$ we find the minimum norm but know nothing
about uniqueness of the minimizing function.

First note that in a normed linear space $X$ with norm $\norm{\cdot}$,
if $x_i\in X$ are linearly independent and $z\in X$, then the problem
of finding $a_i\in\R$ to minimize $\norm{z-a_1x_1-a_2x_2-\cdots-a_nx_n}$
has a solution for each $n\in\N$.  This is called the problem of
best approximation.  For example, \cite[Theorem~7.4.1]{davis}.
Whether this problem has a unique solution depends on the notion of a {\it strictly convex} normed
linear space:
$X$ is strictly convex if for all $x,y\in X$ with $\norm{x}=\norm{y}=1$ and
$x\not=y$ we have $\norm{(x+y)/2}<1$.
Geometrically,
this means the surface of a ball contains no line segments.
It is known that for $1<p<\infty$ the spaces $L^p([-1,1]\times[-1,1])$ are
strictly convex and are not strictly convex if $p=1$ or if  $p=\infty$.  See
\cite[p.~112, exercise~3]{rudin}.  If the elements $x_i$ are linearly
independent in $X$, and $X$ is strictly convex, then the best approximation
problem has a unique solution \cite[Theorem~7.5.3]{davis}.
\begin{theorem}\label{theoremminimizers}
	Define $\phi\fn[a,b]\times[c,d]\to\R$ by $\phi(x,y)=xy+\alpha(x)+\beta(y)$
	where $\alpha$ and $\beta$ are functions of one variable in $L^q([-1,1])$.
	The minimum of $\norm{\phi}_q$, by varying $\alpha$ and $\beta$, is 
	$$
	\norm{\phi}_q=\left\{\begin{array}{cl}
	\left(\frac{2}{q+1}\right)^{2/q}, & 1< q<\infty\\
	1, & q=1 \text{ or } \infty.
	\end{array}
	\right.
	$$
	If $1< q<\infty$, then the unique minimum is given by $\phi(x,y)=(x-m_1)(y-m_2)$ where
	$m_1$ is the midpoint of $[a,b]$, and $m_2$ is
	the midpoint of $[c,d]$.  If $q=1$, or $q=\infty$,
	the minimum is achieved by more than
	one function, but $\norm{\phi}_\infty=1$ with $\phi(x,y)=(x-m_1)(y-m_2)$.
\end{theorem}
\begin{proof}
	It suffices to consider $[a,b]\times[c,d]=[-1,1]\times[-1,1]$, and then a
	linear transformation can be used to map the unit square onto $[a,b]\times[c,d]$.
	
	Let $\psi(x,y)=xy$.
	
	If $1< q<\infty$ then the $q$-norm is strictly convex.  By the paragraph preceding the
	theorem, this means that if $x_i$ are fixed linearly independent functions
	in $L^q([-1,1]^2)$, then for each $n\in\N$ the problem of choosing $a_i\in\R$ to minimize
	$\norm{\psi+a_1x_1+\ldots +a_nx_n}_q$ has a unique solution.  In our problem, the functions $x_i$ 
	are functions of one variable.  We first need a result on linear independence.
	
	Suppose $\alpha_e$, $\alpha_o$, $\beta_e$, and $\beta_o$ are, respectively, non-constant even and odd functions
	of one variable.  We claim that the set of functions $\{\alpha_e(s), \alpha_o(s), \beta_e(t),\\ \beta_o(t)\}$ is
	linearly independent on $[-1,1]^2$.  Suppose 
	$\lambda_1\alpha_e(s)+ \lambda_2\alpha_o(s)+ \lambda_3\beta_e(t)+\lambda_4\beta_o(t)=0$ for all
	$(s,t)\in [-1,1]^2$ for some constants $\lambda_i$.  Then
	$\lambda_1\alpha_e(s)+ \lambda_2\alpha_o(s)=- \lambda_3\beta_e(t)-\lambda_4\beta_o(t)$.
	Since $s$ and $t$ can be varied independently this shows existence of a constant $k$ so that
	$\lambda_1\alpha_e(s)+ \lambda_2\alpha_o(s)=- \lambda_3\beta_e(t)-\lambda_4\beta_o(t)=k$ for
	all $s$ and $t$.  Let $s\not=0$.  Then $\lambda_1\alpha_e(s)- \lambda_2\alpha_o(s)=k$.  
	Adding gives $\lambda_1\alpha_e(s)=k$. Since $\alpha_e$ is not constant we must have 
	$\lambda_1=k=0$.  Subtracting the equations now gives $\lambda_2\alpha_o(s)=0$ and $\alpha_2$ is not 
	constant so $\lambda_2=0$.  Similarly, $\lambda_3=\lambda_4=0$ and the functions are linearly 
	independent.
	
	With the functions $\alpha_e$, $\alpha_o$, $\beta_e$ and $\beta_o$ fixed as above consider the
	expression 
	\begin{align*}
	&\norm{\psi+a_1\alpha_e+a_2\alpha_o+a_3\beta_e+a_4\beta_o}_q^q\\
	&\quad=\int_{-1}^1\int_{-1}^1\abs{st+a_1\alpha_e(s)+a_2\alpha_o(s)+a_3\beta_e(t)+a_4\beta_o(t)}^q\,dt\,ds,
	\end{align*}
	where $a_1$, $a_2$, $a_3$, $a_4$ are the unique constants that give the minimum.  Changing variables
	$(s,t)\mapsto (-s,t)$ in the integral gives
	$$
	\norm{\psi+a_1\alpha_e+a_2\alpha_o+a_3\beta_e+a_4\beta_o}_q^q=
	\norm{\psi-a_1\alpha_e+a_2\alpha_o-a_3\beta_e-a_4\beta_o}_q^q.
	$$
	But the coefficients are unique so $a_1=-a_1$, $a_3=-a_3$ and $a_4=-a_4$.  Hence, these coefficients
	are $0$.
	The change of variables $(s,t)\mapsto (s,-t)$ in the integral now shows $a_2=0$.  Therefore,
	for any set of fixed even and odd functions of one variable the minimum of  
	$\norm{\psi+a_1\alpha_e+a_2\alpha_o+a_3\beta_e+a_4\beta_o}_q$ is $\norm{\psi}_q$.
	
	Now we show that we get the same result when we vary the functions.  Suppose $\alpha$ and $\beta$ are
	any fixed functions in $L^q([-1,1])$.  The even part of $\alpha$ is $\alpha_e(s)=(\alpha(s)+\alpha(-s))/2$
	and the odd part is $\alpha_o(s)=(\alpha(s)-\alpha(-s))/2$.  Similarly with $\beta$.  Again, using the
	convention that $\alpha$ functions are evaluated at the first variable and $\beta$ functions at
	the second variable, we have
	\begin{eqnarray*}
		\min_{c_1,c_2\in\R}\norm{\psi+c_1\alpha+c_2\beta}_q & = & \min_{c_1,c_2\in\R}\norm{\psi+c_1\alpha_e
			+c_1\alpha_o+c_2\beta_e+c_2\beta_o}_q\\
		& \geq & \min_{a_i\in\R}\norm{\psi+a_1\alpha_e+a_2\alpha_o+a_3\beta_e+a_4\beta_o}_q=\norm{\psi}_q.
	\end{eqnarray*}
	But taking $c_1=c_2=0$ gives $\norm{\psi}_q$ in $\norm{\psi+c_1\alpha+c_2\beta}_q$ so this is its
	minimum as well.
	
	Suppose there were functions $\xi,\eta\in L^q([-1,1])$ so that if $\xi$ is evaluated at the first
	variable and $\eta$ is evaluated at the second variable then $\norm{\psi+\xi+\eta}_q<\norm{\psi}_q$.
	Then 
	$$
	\norm{\psi+\xi+\eta}_q=\norm{\psi+1\xi+1\eta}_q\geq \min_{c_1,c_2\in\R}\norm{\psi+c_1\xi+c_2\eta}_q
	=\norm{\psi}_q.
	$$
	This contradiction shows that
	$$
	\min_{\alpha,\beta\in L^q([-1,1])}\norm{\psi+\alpha+\beta}_q=\norm{\psi}_q.
	$$
	
	The norm is computed following \eqref{holderst}.
	
	Now consider $q=\infty$.  The maximum of $\psi(x,y)=xy$ on $[-1,1]\times[-1,1]$ is $\psi(1,1)=\psi(-1,-1)=1$
	and the minimum is $\psi(1,-1)=\psi(-1,1)=-1$.  Hence, $\norm{\psi}_\infty=1$.  For
	$\phi(s,t)=st+\alpha(s)+\beta(t)$ to have $\norm{\phi}_\infty\leq\norm{\psi}_\infty$,
	we must have
	\begin{eqnarray}
	\alpha(1)+\beta(1) & \leq & 0,\label{ab1}\\
	\alpha(-1)+\beta(-1) & \leq & 0,\label{ab2}\\
	\alpha(1)+\beta(-1) & \geq & 0,\label{ab3}\\
	\alpha(-1)+\beta(1) & \geq & 0.\label{ab4}
	\end{eqnarray}
	This is because the maximum of $\psi$ is positive, and the minimum is negative.  And,
	\begin{eqnarray*}
		\eqref{ab1} \text{ and } \eqref{ab3} & \text{give} & \beta(1)-\beta(-1)\leq 0,\\
		\eqref{ab2} \text{ and } \eqref{ab4} & \text{give} & \beta(-1)-\beta(1)\leq 0;
	\end{eqnarray*}
	hence $\beta(1)=\beta(-1)$.  Similarly, $\alpha(1)=\alpha(-1)$.  Equations \eqref{ab1} and
	\eqref{ab3} now show $0\leq\alpha(1)+\beta(1)\leq 0$ and so $\alpha(1)+\beta(1)=0$.  We then
	get $\alpha(1)=\alpha(-1)=-\beta(1)=-\beta(-1)$.  But then
	\begin{eqnarray*}
		\phi(1,1) & = & 1+\alpha(1)+\beta(1)=1\\
		\phi(-1,-1) & = & 1+\alpha(-1)+\beta(-1)=1\\
		\phi(1,-1) & = & -1+\alpha(1)+\beta(-1)=-1\\
		\phi(-1,1) & = & -1+\alpha(-1)+\beta(1)=-1.
	\end{eqnarray*}
	This shows $\norm{\phi}_\infty\geq 1 =\norm{\psi}_\infty$, so
	$\min_{\alpha,\beta}\norm{\phi}_\infty=\norm{\psi}_\infty=1$.
	
	Now consider $q=1$.  Given $\epsilon>0$, for each $\alpha,\beta\in L^1([-1,1])$, there are continuous
	functions $\alphabar,\betabar$ such that 
	$\abs{\norm{\psi+\alpha+\beta}_1-
		\norm{\psi+\alphabar+\betabar}_1}<
	\epsilon$.  Then $\alphabar,\betabar\in L^q([-1,1])$ for each
	$1\leq q\leq \infty$ so
	\begin{eqnarray*}
		\norm{\psi+\alpha+\beta}_1 & \geq & \norm{\psi+\alphabar+\betabar}_1
		-\epsilon\\
		& = & \lim_{q\to 1^+}\norm{\psi+\alphabar+\betabar}_q-\epsilon\\
		& \geq & \lim_{q\to 1^+}\norm{\psi}_q-\epsilon\\
		& = & \lim_{q\to 1^+}\left(\int_{-1}^1\int_{-1}^1\abs{st}^q\,dt\,ds\right)^{1/q}-\epsilon\\
		& = & \lim_{q\to 1^+}\left(2\int_{0}^1s^q\,ds\right)^{2/q}-\epsilon\\
		& = &  \lim_{q\to 1^+}\left(\frac{2}{q+1}\right)^{2/q}-\epsilon\\
		& = & 1-\epsilon.
	\end{eqnarray*}
	Therefore, since $\epsilon>0$ is arbitrary,
	$$
	\min_{\alpha,\beta\in L^1([-1,1])}\norm{\psi+\alpha+\beta}_1\geq 1.
	$$
	But,
	$$
	\norm{\psi +0\alpha+0\beta}_1
	=\int_{-1}^1\int_{-1}^1\abs{st}\,dt\,ds
	=4\left(\int_0^1s\,ds\right)^2=1.
	$$
	Hence, 
	$$
	\min_{\alpha,\beta\in L^1([-1,1])}\norm{\psi+\alpha+\beta}_1= 1.
	$$
	
	An example that shows the minimizing function is not unique when
	$q=\infty$ is $\phi(s,t)=st-\abs{s}+\abs{t}$.  The gradient 
	does not vanish in any of the four open regions $(0,1)\times(0,1)$,
	$(-1,0)\times(0,1)$, $(-1,0)\times(-1,0)$ or $(0,1)\times(-1,0)$.
	The extreme values are then on the $s$-axis for $\abs{s}\leq 1$, on the $t$-axis
	for $\abs{t}\leq 1$, on
	one of the line segments given by
	$\abs{s}=1$, or on one of the line segments given by $\abs{t}=1$.  It is then seen that the 
	maxima and minima on these line segments are $1$ and $-1$.  Hence,
	$\norm{\phi}_\infty=1$.  Further examples with unit norm can
	be obtained by considering $\phi(s,t)=st\pm u\abs{s}^v\mp u\abs{t}^v$
	for $u,v>0$.  A linear transformation then maps the unit square onto
	$[a,b]\times[c,d]$.
\end{proof}
We do not know of an example of non-uniqueness of the minimizing
function when $q=1$.

An approach to the proof for $q=2$ that does not require facts about the uniform
convexity of the norm is the following.  Note that
\begin{align*}
&\int_{-1}^1\int_{-1}^1\abs{st+\alpha(s)+\beta(t)}^2\,dt\,ds\\
&\qquad=\int_{-1}^1\int_{-1}^1\left\{s^2t^2+2st\alpha(s)+2st\beta(t)+[\alpha(s)+\beta(t)]^2\right\}\,dt\,ds\\
&\qquad=\int_{-1}^1\int_{-1}^1\left\{s^2t^2+[\alpha(s)+\beta(t)]^2\right\}\,dt\,ds.
\end{align*}
The norm of $\phi$ is then minimized when $\alpha(s)=-\beta(t)$.  Then
$\alpha$ and $\beta$ are both
constant so the minimizer is $\phi(s,t)=st$.

\section{Acknowledgments}

The first author was partially funded by a work/study grant from the University of the Fraser Valley.


\begin{thebibliography}{99}
	
\bibitem{benedetto}
J.J. Benedetto and W. Czaja, {\it Integration and modern analysis},
Boston, Birkh\"auser, 2009. 
\bibitem{cruzuribeJIPAM}
D. Cruz-Uribe and C.J. Neugebauer, {\it Sharp error bounds for the trapezoidal
	rule and Simpson's rule}, JIPAM. J. Inequal. Pure Appl. Math. {\bf 3}(2002),
Article 49, 22~pp.
\bibitem{currier}
A.E. Currier, {\it Proof of the fundamental theorems on second-order cross partial derivatives},
Trans. Amer. Math. Soc. {\bf 35}(1933), 245--253.
\bibitem{davis}
P.J. Davis, {\it Interpolation and approximation}, New York, Dover,
1975.
\bibitem{haber}
S. Haber, {\it Numerical evaluation of multiple integrals}, SIAM Rev. {\bf 12}(1970), 
481--526.
\bibitem{mcleod}
R.M. McLeod, {\it The generalized Riemann integral}, Washington,
Mathematical Association of America, 1980.
\bibitem{mikeladze}
Sh. E. Mikeladze, {\it Numerical methods of mathematical
	analysis}, Moscow, State Publishing House of Technical-Theoretical
Literature, 1953.
\bibitem{rudin}
W. Rudin, {\it Real and complex analysis}, New York, McGraw-Hill, 1987.
\bibitem{sard}
A. Sard, {\it Linear approximation},
Providence, Rhode Island, American Mathematical Society, 1963.
\bibitem{stroud}
A.H. Stroud, {\it Approximate calculation of multiple integrals},
Englewood Cliffs, New Jersey, Prentice-Hall, 1971.
\bibitem{talvilawiersmaaejm}
E. Talvila and M. Wiersma, {\it Simple derivation of basic quadrature formulas},
Atl. Electron. J. Math. {\bf 5}(2012), 47--59.
\bibitem{ueberhuber}
C.W. Ueberhuber, {\it Numerical computation, vol. II}, Berlin,
Springer--Verlag, 1997.

\end{thebibliography}
\end{document}